\documentclass[USenglish]{article}

\usepackage[square,comma,numbers]{natbib}
\usepackage{amsthm}
\usepackage{leftidx}
\usepackage{mathrsfs}
\usepackage[utf8]{inputenc}
\usepackage[small]{dgruyter}

\def\noi{\noindent}

\renewcommand{\thefootnote}{{}}

\def\vp{\varphi}

\def\s1deg{S^1\text{\rm -Deg\,}}

\newtheorem{theorem}{Theorem}[section]

\newtheorem{lemma}[theorem]{Lemma}
\newtheorem{corollary}[theorem]{Corollary}

\newtheorem{definition}[theorem]{Definition}
\newtheorem{remark}[theorem]{Remark}

\newtheorem{remark-definition}[theorem]{Remark and Definition}

\begin{document}

  \author[1]{Z Balanov}
  \author*[2]{E Hooton	}
  \author[3]{A Murza} 
 
  \affil[2]{Department of Mathematical Sciences, University of Texas at Dallas, Richardson, Texas, 75080 US.
  Email: exh121730@utdallas.edu}
  \affil[1]{Department of Mathematical Sciences, University of Texas at Dallas, Richardson, Texas, 75080 US.}
  \affil[3]{Institute of Mathematics Simion Stoilow of the Romanian Academy, P.O. Box 1--764, RO--014700, Bucharest, Romania.}
  
  \title{Periodic Solutions of vdP and vdP-like Systems on $3$--Tori}

  \abstract{Van der Pol equation (in short, vdP) as well as many its non-symmetric generalizations (the so-called van der Pol-like oscillators (in short, vdPl)) serve as nodes in coupled networks modeling real-life phenomena. Symmetric properties of periodic regimes  of networks  of vdP/vdPl  depend on symmetries of coupling. In this paper, we consider $N^3$ identical 
  vdP/vdPl oscillators arranged in a cubical lattice, where opposite faces are identified in the same way as for a $3$-torus. Depending on which nodes impact the dynamics of a given node, we distinguish between $\mathbb D_N \times \mathbb D_N \times \mathbb D_N$-equivariant systems and their 
  $\mathbb Z_N \times \mathbb Z_N \times \mathbb Z_N$-equivariant  counterparts. In both settings, the local equivariant Hopf bifurcation together with the global 
  existence of periodic solutions with prescribed period and symmetry, are studied.
  The methods used in the paper are based on the results rooted in both equivariant degree theory and (equivariant) singularity theory.}

  \startpage{1}

\maketitle

\section{Introduction}\
{\bf (a) Subject and Goal.}
The van der Pol equation (in short, vdP, cf. \eqref{vdP}), originally introduced in \citep{VdP} to study stable oscillations in electrical circuits, is very often considered as a starting point of applied nonlinear dynamics.  An important feature of equation   \eqref{vdP} is that it respects the antipodal symmetry. Different generalizations of vdP, which break this symmetry, have been considered by many authors in connection to a wide range of applied problems (in what follows, we will call these generalizations van der Pol-like equations (in short, vdPl)). Examples of particular importance include the FitzHugh-Nagumo model  (see, for example, \citep{FitzHugh}) and the realistic kinetic model of the chlorite-iodide-malonic acid reaction (see, for example, \citep{Boissonade}). To be more concrete about the importance of considering vdPl systems with quadratic terms, we refer to \citep{VdPlike} and the references therein.

In real life models, vdP as well as vdPl serve as nodes in coupled networks.
Symmetries of the coupling have an impact on the symmetries of the actual dynamics.
In this paper,  we will consider $N^3$ oscillators arranged in a cubical lattice, where opposite faces are identified in the same way as for a $3$-torus.
In such a configuration, two aspects of the coupling are important:
(i)  which nodes impact the dynamics of a given node (which we will call {\it coupling topology}), and
(ii) how a neighboring node impacts a given node (which we will call {\it coupling structure}).
For the coupling topology, we consider the following two cases:
(i)  all $6$ neighbors of a given node impact on that node's dynamics (we call such a coupling {\it bi-directional});
(ii) only $3$ neighbors of a given node impact on that node's dynamics (we call such a coupling {\it uni-directional}).
In the case of bi-directional coupling, the system respects a natural action of $\mathbb D_N\times \mathbb D_N\times \mathbb D_N$, while in the case of uni-directional coupling the symmetry generated by $\kappa$ is destroyed, hence the total symmetry group is $\mathbb Z_N \times \mathbb Z_N \times \mathbb Z_N$  (cf. \citep{3d-torus-applications} and references therein).  We will distinguish between two linear coupling structures, namely,  for a given node, either the $x$-variable of a neighbor or the $y$-variable of a neighbor is coupled to the $x$ variable of the specified node (compare \eqref{nvdP_2d} with \eqref{nvdPl_2d}). We will call these $x$-coupling and $y$-coupling respectively.

The {\it goal} of this paper is three-fold, namely, in the settings introduced above, we will: (i) establish the {\it occurrence} of the Hopf bifurcation, classify symmetric properties of the bifurcating branches and estimate their number; (ii) study {\it stability} of the corresponding periodic solutions, and (iii) investigate the {\it existence} of periodic solutions with prescribed period and symmetry.

\medskip

{\bf (b) Results.} Keeping in mind a wide spectrum of potential applications in natural sciences and engineering, it is worthy to study the above mentioned problems (occurrence, stability and existence) in all possible settings. The usual dilemma of keeping a balance between Scylla of completeness and Charybdis of reasonable size of the manuscript,  resulted in our paper as follows:

\medskip

(i) Although, using the methods developed in this paper,  the  occurrence/multiplicity estimates/symmetry classification of the Hopf bifurcation can be established for any combination of  bi-directional/uni-directional  vdP/vdPl $x$-coupled/$y$-coupled systems,
we only treat two cases, namely, that of bi-directionally $y$-coupled vdP oscillators and uni-directionally $x$-coupled vdPl oscillators.

(ii) We provide an instability result for   bi-directionally $y$-coupled vdP oscillators and stability results for  uni-directionally $x$-coupled vdPl oscillators.

(iii) We establish the  existence of periodic solutions with prescribed period and symmetry only in the case of bi-directionally $y$-coupled vdP oscillators.

\medskip
{\bf (c) Methods.} The methods used in this paper are based on the results rooted in both equivariant degree theory and (equivariant) singularity theory.
To be more specific:

\medskip

(i) To treat the occurrence/multiplicity estimates/symmetric classification  of the Hopf bifurcation, we appeal to the abstract results presented in \citep{AED} (see also \citep{BK-chapter, Sliding}).

(ii) The stability results are obtained in the framework of the theory presented in \citep{Guckenheimer} (see also \citep{GolSteShef}).

(iii) The main ingredient to establish the existence of periodic solutions with prescribed period and symmetry is Theorem \ref{theom-abstract-Hirano-Ryb} which was presented in
\citep{AED} and \citep{BFK} (see also \citep{Hirano-Rybicki}).

\smallskip

For the representation theory background, we refer to \citep{Brocker-tomDieck,Serre}.
\medskip

{\bf  (d) Overview.} After the Introduction, the paper is organized as follows. In Section \ref{onetorus_theor}, we formulate main results of the paper. Section \ref{sect:occurrence-result}
is devoted to the proof of the main occurrence/multiplicity/symmetry results (Theorems \ref{thm_vdp} and \ref{thm_vdpl}). The proof (see Subsection \ref{Proof-Theor-occurr}) is based on an abstract Theorem \ref{abs_therom} and equivariant spectral information collected in Subsections \ref{sec_decop}--\ref{iso_sec}. We believe that several algebraic observations related to the computation of maximal twisted orbit types in complexifications of tensor product representations (see Subsection \ref{iso_sec}) may be interesting in their own. Section \ref{sect-stability} contains the proof of the instability result for system \eqref{nvdP_2d}  (see Theorem \ref{Dn_instability}) and stability result for system \eqref{nvdPl_2d} (see Theorem \ref{Zn_stab}). The proof follows the standard lines (see \citep{Guckenheimer}, Theorem 3.4.2, and \citep{GolSteShef}),  and combines the spectral equavariant data from Subsections \ref{sec_decop}--\ref{iso_sec} with the computations of the first Lyapunov coefficient from Subsection \ref{subsec:Lyapunov-coef}. The proof of the existence result (see Theorem  \ref{Existence-Dihedral}) is given in Section
\ref{sec:Existence}
where one can also find an adapted version of the abstract existence result given in \citep{AED}, Theorem 12.7,  and \citep{BFK} (cf. Theorem  \ref{theom-abstract-Hirano-Ryb}).
We conclude with a short Appendix where  several symbols frequently used in this paper to denote some twisted groups are explained (cf. \citep{AED}).

\medskip
{\bf (e) Acknowledgements.} The first two authors acknowledge the support from National Science Foundation through grant DMS-413223. The first author is grateful for the support of the Gelbart Research Institute for mathematical sciences at Bar Ilan University. The third author acknowledges a postdoctoral BITDEFENDER fellowship from
the Institute of Mathematics Simion Stoilow of the Romanian Academy, Contract of Sponsorship No. 262/2016.

\section{Main Results}\label{onetorus_theor}
In this paper, we are interested in networks of identical vdP oscillators \begin{equation}\label{vdP}
\begin{array}{ll}
        \dot{x}&=\nu(ax-x^3)-y\\
        \dot{y}&=bx
    \end{array}
\end{equation}
and vdP-like oscillators
\begin{equation}\label{vdPl}
\begin{array}{ll}
        \dot{x}&=-y-x^3-x^2+ax\\
        \dot{y}&=bx
    \end{array}
\end{equation}
coupled in the symmetric configuration of a three-dimensional torus. To be more precise, we consider $N^3$ oscillators, where $N$ is an {\it odd} number, with both bi-directional coupling
\begin{equation}\label{nvdP_2d}
\left\{
    \begin{aligned}
        \dot{x}_{(\alpha,\beta,\gamma)}= &
        \nu(ax_{(\alpha,\beta,\gamma)}-x_{(\alpha,\beta,\gamma)}^3) -y_{(\alpha,\beta,\gamma)}
        \\
        &+\delta \left(2y_{(\alpha,\beta,\gamma)}-y_{(\alpha+1,\beta,\gamma)}-y_{(\alpha-1,\beta,\gamma)}\right)\\& +\zeta \left(2y_{(\alpha,\beta,\gamma)}-y_{(\alpha,\beta+1,\gamma)}-y_{(\alpha,\beta-1,\gamma)}\right)
        \\
        &+\varepsilon \left(2y_{(\alpha,\beta,\gamma)}-y_{(\alpha,\beta,\gamma+1)}y_{(\alpha,\beta,\gamma-1)}\right)
        \\
        \dot{y}_{(\alpha,\beta,\gamma)}=& \; bx_{(\alpha,\beta,\gamma)}
    \end{aligned}
    \right.
\end{equation}
and uni-directional coupling
\begin{equation}\label{nvdPl_2d}
\left\{
    \begin{aligned}
        \dot{x}_{(\alpha,\beta,\gamma)}= &
        -y_{(\alpha,\beta,\gamma)}-x_{(\alpha,\beta,\gamma)}^3-x_{(\alpha,\beta,\gamma)}^2+ax_{(\alpha,\beta,\gamma)}
        \\
        &+\delta \left(x_{(\alpha,\beta,\gamma)}-x_{(\alpha+1,\beta,\gamma)}\right)+\zeta \left(x_{(\alpha,\beta,\gamma)}-x_{(\alpha,\beta+1,\gamma)}\right)
        \\
        &+\varepsilon \left(x_{(\alpha,\beta,\gamma)}-x_{(\alpha,\beta,\gamma+1)}\right)
        \\
        \dot{y}_{(\alpha,\beta,\gamma)}=& \; bx_{(\alpha,\beta,\gamma)}
    \end{aligned}
    \right.
\end{equation}
Here $x,y \in \mathbb R^{N^3}$ and their entries are indexed by the triple $(\alpha,\beta,\gamma)$ where $\alpha,\beta,\gamma \in \{1,\cdots,N\}$, $\delta, \zeta,\varepsilon\in \mathbb R$ and $\nu,b>0$.
\begin{remark}
{\rm To avoid distinctions occurring due to the parity of $N$, we will only consider the case when $N$ is odd.}
\end{remark}

\begin{definition}\label{spa_temp}{\rm 
We will say that a periodic function $x:\mathbb R \to U$ with period $T$ has a (spatio-temporal) symmetry $H< G\times S^1$, if for every $(g,e^{i\theta})\in H$ and for every $t$,
$$
g\cdot x(t+\theta T/2\pi)= x(t).
$$ 
}
\end{definition}

\begin{theorem}\label{thm_vdp}
For each fixed ${\bf t}= (t_1,t_2,t_3)$, where $t_1,t_2,t_3 \in  \{1,\cdots,n\}$, put
\begin{equation}\label{Kt_def1}
K_{\bf t}= 1 + 2\delta(1-\cos(2\pi t_1/N))+2\zeta(1-\cos(2\pi t_2/N))+2\varepsilon(1-\cos(2\pi t_3/N)).
\end{equation}
If $K_{\bf t} > 0$, then system \eqref{nvdP_2d} undergoes Hopf bifurcation as $a$ passes $0$.
Furthermore, for each $(H^\varphi) \in \mathcal S({\bf t})$,  there exist $\frac{8N^3}{|H_1\times H_2\times H_3|}$ branches of bifurcating non-constant periodic solutions of \eqref{nvdP_2d} with limit frequency $\omega_{\bf t}= \sqrt{bK_{\bf t}}$ and   minimal symmetry $(H^\varphi) $  (here $\mathcal S (\bf t)$ is the set of spatio-temporal symmetries associated to {\bf t} which is described in Section \ref{iso_sec}, formula \eqref{desc_iso}).
\end{theorem}

\begin{theorem}\label{thm_vdpl}
For each fixed ${\bf t}= (t_1,t_2,t_3)$, where $t_1,t_2,t_3 \in  \{1,\cdots,N\}$, system \eqref{nvdPl_2d} undergoes Hopf bifurcation as $a$ passes
\begin{equation}\label{eq:bifur-points}
a^z_{\bf t}= \delta(1-\cos(2t_1\pi/N))+\zeta(1-\cos(2t_2\pi/N))+\varepsilon(1-\cos(2t_3\pi/N)).
\end{equation}
Furthermore, there exist a branch of bifurcating non-constant periodic solutions of \eqref{nvdPl_2d} with symmetry $(\mathbb Z_n\times \mathbb Z_n \times \mathbb Z_n)^{(t_1,  t_2, t_3)}$ and limit period
\begin{align*}
P^1_{\bf t}&=\Big|\frac{4\pi}{H+\sqrt{H^2+4b}}\Big|,
\end{align*}
and a branch with the same symmetry and limit period
\begin{align*}
P^2_{\bf t}&=\Big|\frac{4\pi}{H-\sqrt{H^2+4b}}\Big|,
\end{align*}
where $ H= \delta\sin((2t_1\pi/N))+\zeta\sin(2t_2\pi/N))+\varepsilon(\sin(2t_3\pi/N))$ and the symbol  $(\mathbb Z_n \times \mathbb Z_n \times \mathbb Z_n)^{(t_1,  t_2, t_3)}$ is described in Appendix.
\end{theorem}
\begin{remark}
{\rm We do not guarantee that $P_{\bf t}^1$ or $P_{\bf t}^2$ are minimal periods. For this reason, it is possible for a single solution to have both $P_{\bf t}^1$ and $P_{\bf t}^2$ as a period, in which case we only guarantee the existence of a single branch. This can occur in the case when $P_{\bf t}^1/P_{\bf t}^2\in\mathbb Q$.}
\end{remark}

\begin{theorem}\label{Dn_instability}
Put $\theta_N:= \frac{(N-1) \pi}{N}.$ Suppose $k_1\delta+k_2\zeta+k_3\varepsilon$ is less than
$$
\frac{1}{2(\cos(\theta_N)-1)}
$$
for some $k_1,k_2,k_3 \in \{0,1\}$. Then, all branches of bifurcating non-constant periodic solutions of \eqref{nvdP_2d} guaranteed by Theorem \ref{thm_vdp} are unstable.
\end{theorem}

\begin{theorem}\label{Zn_stab}
For any fixed $\delta,\zeta$ and $\varepsilon$, the equilibrium of system \eqref{nvdPl_2d}  is stable for $a<a^*$ and unstable for $a>a^*$, where $a^*$ is described in
Table~\ref{tabelaBifurca}. Furthermore, the \textbf{\textit {fully synchronized}} branch of  system \eqref{nvdPl_2d}, born at $a=0$, is stable if and only if $\delta,\zeta$ and $\varepsilon$ are negative.
\end{theorem}

\begin{table}[h]
\begin{center}
\begin{tabular}{c|c|c|c|c}
\hline
sign$(\delta)$&sign$(\zeta)$& sign$(\epsilon)$& $a_*$ & ${\bf t} = (t_1,t_2,t_3) $ \\
\hline
-&-&-&0&$(0,0,0)$\\[0.25em]
+&-&-&$\delta\left(1-\cos\left(\theta_N\right)\right)$&$\left(\frac{N-1}{2},0,0\right)$\\[0.25em]
-&+&-&$\zeta\left(1-\cos\left(\theta_N\right)\right)$&$\left(0,\frac{N-1}{2},0\right)$\\[0.25em]
-&-&+&$\epsilon\left(1-\cos\left(\theta_N\right)\right)$&$\left(0,0,\frac{N-1}{2}\right)$\\[0.25em]
+&+&-&$(\delta+\zeta)\left(1-\left(\cos\theta_N\right)\right)$&$\left(\frac{N-1}{2},\frac{N-1}{2},0\right)$\\[0.25em]
+&-&+&$(\delta+\epsilon)\left(1-\left(\cos\theta_N\right)\right)$&$\left(\frac{N-1}{2},0,\frac{N-1}{2}\right)$\\[0.25em]
-&+&+&$(\zeta+\epsilon)\left(1-\left(\cos\theta_N\right)\right)$&$\left(0,\frac{N-1}{2},\frac{N-1}{2}\right)$\\[0.25em]
+&+&+&$(\delta+\zeta+\epsilon)\left(1-\left(\cos\theta_N\right)\right)$&$\left(\frac{N-1}{2},\frac{N-1}{2},\frac{N-1}{2}\right)$\\[0.25em]
\hline
\end{tabular}\end{center}
\caption{Details of the Hopf bifurcations from a stable equilibrium(here $\theta_N= \frac{(N-1)\pi}{N}$)}
\label{tabelaBifurca}
\end{table}

\begin{theorem}\label{Existence-Dihedral}
Assume that $K_{\bf t} \neq 0$  for any ${\bf t}=(t_1,t_2,t_3)$ (cf. \eqref{Kt_def1}). Choose $p \not\in\big\{\frac{2\pi(2k-1)}{\sqrt{bd}}\; : \; k \in \mathbb Z, \; d>0,\; d= K_{\bf t} \; \text{for some} \; {\bf t}\big\}$. Then, for each {\bf t} with $K_{\bf t} > (\frac{2\pi}{p})^2$, there exists a value of $\nu$ such that for each  $(H^{\varphi}) \in \mathcal S({\bf t})$,  system \eqref{nvdP_2d} admits $\frac{8N^3}{|H_1\times H_2\times H_3|}$  $p$-periodic solutions with minimal symmetry $(H^{\varphi})$ (here $\mathcal S (\bf t)$ is the set of symmetries associated to {\bf t} which is described in Section \ref{iso_sec}, formula \eqref{desc_iso}).
\end{theorem}

\section{Equivariant spectral  data and first Lyapunov coefficient}
\subsection{Isotypical decomposition of the phase space}\label{sec_decop}

Although \eqref{nvdP_2d} and \eqref{nvdPl_2d} have different symmetry groups, they have the same phase space ($\mathbb R^{N^3}$) as  a geometric set. 
Since it is usually unambiguous, we will use the same notation for the representations of $G_1:= \mathbb Z_N \times \mathbb Z_N \times \mathbb Z_N$ and $G_2:= \mathbb D_N \times \mathbb D_N \times \mathbb D_N$. Put $V := \mathbb R^{N^3}$ and denote by $W := V \oplus V$ the phase space of systems \eqref{nvdP_2d} and \eqref{nvdPl_2d}.
To describe spatial symmetries of system \eqref{nvdPl_2d}, we will consider $G_1= \mathbb Z_N \times \mathbb Z_N \times \mathbb Z_N$ as a subgroup of $ S^1\times S^1\times S^1$ and define the $G_1$-action on $V$ by specifying how each of its generators acts, namely:

\begin{align*}
((e^{\frac{2\pi i}{N}},1,1)\cdot x)_{(\alpha,\beta,\gamma)}= x_{(\alpha+1,\beta,\gamma)}\\
((1,e^{\frac{2\pi i}{N}},1)\cdot x)_{(\alpha,\beta,\gamma)}= x_{(\alpha,\beta+1,\gamma)}\\
((1,1,e^{\frac{2\pi i}{N}})\cdot x)_{(\alpha,\beta,\gamma)}= x_{(\alpha,\beta,\gamma+1)}.
\end{align*}

Here $+$ is taken modulo $N$. By direct verification, the right-hand side of system \eqref{nvdPl_2d} is $G_1$-equivariant.
To extend this action to a  $G_2$-action, we need to specify how the remaining generators act, namely:

\begin{align*}
((\kappa,1,1)\cdot x)_{(\alpha,\beta,\gamma)}= x_{(-\alpha,\beta,\gamma)}\\
((1,\kappa,1)\cdot x)_{(\alpha,\beta,\gamma)}= x_{(\alpha,-\beta,\gamma)}\\
((1,1,\kappa)\cdot x)_{(\alpha,\beta,\gamma)}= x_{(\alpha,\beta,-\gamma)}
\end{align*}
where $-$ is again taken modulo $N$. It is clear that the right-hand side of \eqref{nvdP_2d} is $G_2$-equivariant.

 To describe the isotypical decomposition of $V$ as a $G_1$-space, we need to classify all (real) irreducible $G_1$-representations. For each $0\leq t_1,t_2,t_3\leq N-1$, put ${\bf t}:= (t_1,t_2,t_3)$ and denote by $\mathcal V^z_{\bf t}$ an irreducible representation of $G_1$ associated with ${\bf t}$. We have
\begin{equation}\label{Z_triv_rep}
\mathcal V^z_{0,0,0}= \mathbb R
\end{equation}
is the trivial real $G_1$-representation. For $(t_1,t_2,t_3)\neq (0,0,0)$, put
\begin{equation}\label{irr_def}
\mathcal V^z_{\bf t}=\mathbb R^2
\end{equation}
and define the $G_1$-action as follows:
$$
\Big(e^{\frac{2k_1\pi i}{N}},e^{\frac{2k_2\pi i}{N}},e^{\frac{2k_3\pi i}{N}}\Big)\cdot  \begin{bmatrix}
x \\ y
\end{bmatrix} = A  \begin{bmatrix} x \\ y\end{bmatrix},
$$
where
$$
A:= \begin{bmatrix}
\cos\big((k_1t_1+k_2t_2+k_3t_3)\left(\frac{2 \pi}{N}\right)\big) & \sin\big((k_1t_1+k_2t_2+k_3t_3)\left(\frac{2 \pi}{N}\right)\big)\\
-\sin\big((k_1t_1+k_2t_2+k_3t_3)\left(\frac{2 \pi}{N}\right)\big) & \cos\big((k_1t_1+k_2t_2+k_3t_3)\left(\frac{2 \pi}{N}\right)\big)
\end{bmatrix} 
$$
By direct verification, $\mathcal V^z_{\bf t}$ is $G_1$-equivalent to $\mathcal V^z_{-\bf t}$ where $-$ is taken modulo $N$. Hence, there is one one-dimensional trivial representation and $(N^3-1)/2$ non-trivial two-dimensional non-equivalent $G_1$-representations.

For each fixed ${\bf t}$, we define vectors $x^1_{\bf t}$ and $x^2_{\bf t}$ by specifying them component-wisely as follows:
$$
 \big(x^1_{\bf t}\big)_{\alpha,\beta,\gamma}= \cos(\alpha t_1+\beta t_2 + \gamma t_3)
 $$
 $$
  \big(x^2_{\bf t}\big)_{\alpha,\beta,\gamma}= \sin(\alpha t_1+\beta t_2 + \gamma t_3).
  $$
Define a family of subspaces of $V^z_{\bf t}$ by
$$V^z_{\bf t}= \text{span} (x^1_{\bf t}, x^2_{\bf t}).$$
Notice that $V^z_{\bf t}$ is a $G_1$-irreducible component of $V$ and  is $G_1$-equivalent to $\mathcal V^z_{\bf t}$ (cf. \eqref{Z_triv_rep}, \eqref{irr_def}).

\begin{remark}\label{Z_prim_decomp}
{\rm  $V$ admits a primary $G_1$-decomposition which includes every $G_1$-irreducible representation.}
\end{remark}

Let us denote by  $\mathcal U_0^d$ the trivial one-dimensional real $\mathbb D_N$-representation and by $\mathcal U_t^d$ the natural 2-dimensional real $\mathbb D_N$-representation, where the action is defined by

$$
e^{\frac{2k\pi i}{N}}\cdot  \begin{bmatrix}
x\\y\end{bmatrix}= \begin{bmatrix}
\cos\big(kt\left(\frac{2 \pi}{N}\right)\big) & \sin\big(kt\left(\frac{2 \pi}{N}\right)\big)\\
-\sin\big(kt\left(\frac{2 \pi}{N}\right)\big) & \cos\big(kt\left(\frac{2 \pi}{N}\right)\big)
\end{bmatrix} \begin{bmatrix}
x\\y
\end{bmatrix}$$
and
$$
\kappa\cdot \begin{bmatrix}
x\\y
\end{bmatrix} = \begin{bmatrix}
y\\x
\end{bmatrix}
$$
Denote
$$\mathcal V^d_{\bf t}:=
\mathcal U_{t_1}^d \otimes \mathcal U_{t_2}^d \otimes \mathcal U_{t_3}^d. $$
Since $\mathcal U_{t}^d$ is of real type for any $t= 0 ,\cdots, N$, it is easy to see that $\mathcal V_{\bf t}^d$ is a real {\it irreducible} $G_2=(\mathbb D_N \times \mathbb D_N\times \mathbb D_N)$-representation. Furthermore, the dimension of $\mathcal V_{\bf t}^d$ is either $1,2,4$ or $8$ depending on how many non-zero components {\bf t} has.
Put ${\bf t}^\dagger:=(t_1,-t_2,t_3)$, ${\bf t}^\# =(t_1,t_2,-t_3)$ and ${\bf t}^* :=(t_1,t_2,-t_3)$. Put
$$
V_{\bf t}^d := \text{span}(x^1_{\bf t},x^2_{\bf t},x^1_{\bf t^\dagger},x^2_{\bf t^\dagger},x^1_{\bf t^\#},x^2_{\bf t^\#},x^1_{\bf t^*},x^2_{\bf t^*})
$$
By simple but lengthy computations, one can easily show that $V_{\bf t}^d$ is  $G_2$-invariant and equivalent to $\mathcal V^d_{\bf t}$.

\begin{remark}\label{prim_decom}
{\rm  $V$ admits a primary $G_2$-decomposition, however, unlike its decomposition as a $G_1$-space (cf. Remark \ref{Z_prim_decomp}), the decomposition as a $G_2$-space does not include every $G_2$-irreducible representation.}
\end{remark}

\subsection{Equivariant spectral decomposition}\label{sec_eign}
The linearization of system \eqref{nvdP_2d} at the origin restricted to the isotypical component $V^d_{\bf t}\oplus V^d_{\bf t}$ is given by $A^d_{\bf t}(a):V^d_{\bf t}\oplus V^d_{\bf t} \to V^d_{\bf t}\oplus V^d_{\bf t}$, where
\begin{equation}\label{Dn_blocks}
A^d_{\bf t}(a)= \left(
\begin{array}{cc}
\nu a & -K_{\bf t}\\
b & 0
\end{array}
\right) \otimes \text{Id}_{\bf t},
\end{equation}
\begin{equation}\label{Kt_def}
K_{\bf t}=  1+2\delta(1-\cos(2\pi t_1/N))+2\zeta(1-\cos(2\pi t_2/N))+2\varepsilon(1-\cos(2\pi t_3/N))
\end{equation}  and $\text{Id}_{\bf t}$ is the  matrix of the identity operator on $V^d_{\bf t}$ (here $\otimes$ stands for the Kroneker product of matrices). It is clear that the eigenvalues of $A^d_{\bf t}(a)$ are given by
\begin{equation}\label{eig_d_syst}
\lambda_{\bf t}^d(a)= \frac{\nu a\pm \sqrt{(\nu a)^2-4bK_{\bf t}}}{2}.
\end{equation}

\bigskip

The linearization of system \eqref{nvdPl_2d} at the origin restricted to the isotypical component  $V^z_{\bf t}\oplus V^z_{\bf t}$ is given by $A^z_{\bf t}(a):V^z_{\bf t}\oplus V^z_{\bf t} \to V^z_{\bf t}\oplus V^z_{\bf t}$, where
$$A^z_{\bf t}(a)= \left(
\begin{array}{cccc}
H_{\bf t}(a) & G_{\bf t} & -1 & 0\\
-G_{\bf t} & H_{\bf t}(a)& 0 & -1\\
b & 0 & 0 & 0 \\
0 & b & 0 & 0
\end{array}
\right),$$
$$
H_{\bf t}(a) = a-\delta(1-\cos(2\pi t_1/N))-\zeta(1-\cos(2\pi t_2/N))-\varepsilon(1-\cos(2\pi t_3/N))$$ and 
\begin{equation}\label{H_t}
G_{\bf t}=\delta(\sin(2\pi t_1/N))+\zeta(\sin(2\pi t_2/N))+\varepsilon(\sin(2\pi t_3/N)).
\end{equation} To compute the eigenvalues of $A^z_{\bf t}(a)$, we should notice that  $A^z_{\bf t}(a) = {}^\mathbb RL_{\bf t}(a)$, where $L_{\bf t}(a)$ is the complex matrix

\begin{equation}
\label{matL_t}
L_{\bf t}(a)= \left(
\begin{array}{cc}
H_{\bf t}(a)-iG_{\bf t}  & -1\\
b & 0
\end{array}
\right)
\end{equation}
(here the symbol ${}^\mathbb RL_{\bf t}$ stands for the realification of $L_{\bf t}$). It is well-known that $\sigma ({}^\mathbb RL_{\bf t}(a)) = \sigma (L_{\bf t}(a)) \cup \overline{\sigma (L_{\bf t}(a))}.$ It is easy to see that the eigenvalues of $L_{\bf t}(a)$ are given by
\begin{equation}\label{eig_z_sys}
\lambda^{z}_{\bf t}(a)= \frac{H_{\bf t}(a)-iG_{\bf t} \pm \sqrt{(H_{\bf t}(a)-iG_{\bf t})^2-4b}}{2}
\end{equation}
\begin{remark}{\rm For a generic choice of parameters $\gamma, \zeta, \varepsilon$, all eigenvalues $\lambda^{z}_{\bf t}(a)$ are distinct and in the case $\lambda^{z}_{\bf t}(a)$ is purely imaginary, not in resonance.}
\end{remark}

\subsection{Isotropies in  $(V^z_{\bf t})^c$ and $(V^d_{\bf t})^c$}\label{iso_sec}
{\bf Note}: 
For an explanation of all symbols used in this section, see the Appendix.

{(a)} Since $(V^z_{\bf t})$ is a real irreducible $G_1$-representation of complex type, $(V^z_{\bf t})^c$ decomposes into the direct sum of two non-equivalent conjugate irreducible $(G_1 \times S^1)$-representations:
$V_1\oplus \overline V_1$. Since $V_1$ is one-dimensional, it has only two orbit types, namely $(G_1\times S^1)$ and  $(\mathbb Z_n\times \mathbb Z_n \times \mathbb Z_n)^{(t_1,  t_2, t_3)}$.
Similarly, $\overline V_1$ has only   $G_1\times S^1$ and  $(\mathbb Z_n\times \mathbb Z_n \times \mathbb Z_n)^{(-t_1,  -t_2, -t_3)}$.

\bigskip

{(b)}  Let us recall the following
\begin{definition}\label{def-max-orbit}
{\rm  We will call an orbit type $H < G \times S^1$ maximal in $\mathcal U^c_j$ if for any $\widetilde{H} \neq  G \times S^1$ which is an orbit type in $\mathcal U^c_j$, one has $\widetilde{H}<H$.}
\end{definition}

To restrict candidates for maximal isotropies in  $(V^d_{\bf t})^c$, we use the following simple observation.
\begin{lemma}\label{lem:factor-product}
Let $\mathcal G_1$ (resp. $\mathcal G_2$) be a finite group and let $U_1$ (resp. $U_2$)  be a unitary $\mathcal G_1 \times S^1$-respresentation (resp.  $G_2  \times S^1$-respresentation), where
$S^1$ acts on $U_1$ and $U_2$ by complex multiplication.

\medskip

(i) If $H_1^{\varphi_1}$ is a twisted isotropy in $U_1$ and  $H_2^{\varphi_2}$ is a twisted isotropy in $U_2$, then $(H_1 \times H_2)^{(\varphi_1,\varphi_2)}$ is an isotropy in $U_1 \otimes U_2$.

\medskip



(ii) If $(H_1 \leftidx{}{_{K_1}}{\times}_{K_2} H_2)^{(\varphi_1,\varphi_2)}$ is an isotropy in $U_1 \otimes U_2$, then for some $v_1 \in U_1$ and $v_2 \in U_2$, one has $\mathcal G_{v_1} \geq K_1^{\varphi_1}$ and
$\mathcal G_{v_2} \geq K_2^{\varphi_2}$ (cf. \eqref{amal_not}).
\end{lemma}

\begin{proof}
(i) Take $v_1 \in (\mathcal U_1)_{H_1^{\varphi_1}}$ and $v_2 \in (\mathcal U_2)_{H_2^{\varphi_2}}$.
Then, for any  
$$
g:=(h_1,h_2,\varphi_1(h_1)\varphi_2(h_2)) \in (H_1 \times H_2)^{(\varphi_1,\varphi_2)},
$$
one has $g(v_1 \otimes v_2) = v_1 \otimes v_2$, i.e.,
$\mathcal G_{v_1 \otimes v_2} \geq (H_1 \times H_2)^{(\varphi_1,\varphi_2)}$. On the other hand, if $\mathcal G_{v_1 \otimes v_2} \ni
(h_1,h_2,e^{i \theta})$, then $\Big( e^{i \theta} T_{h_1} \otimes T_{h_1} \Big)  v_1 \otimes v_2  =  v_1 \otimes v_2$, which implies that for some $\hat{\theta}$ and $\tilde{\theta}$ with
$\theta = \hat{\theta} + \tilde{\theta}$ one has $e^{i\hat{\theta}} T_{h_1} v _1 = v_1$ and $e^{i\tilde{\theta}} T_{h_2} v _2 = v_2$.  Hence,
$(h_1, e^{i\hat{\theta}}) \in \mathcal G_{v_1} = H_1^{\varphi_1}$, $(h_1, e^{i\tilde{\theta}}) \in \mathcal G_{v_2} =  H_2^{\varphi_2}$ and
$(h_1,h_2,e^{i\theta}) = (h_1,h_2,e^{i\hat{\theta}} \cdot e^{i\tilde{\theta}}) \in (H_1 \times H_2)^{(\varphi_1,\varphi_2)}$
i.e., $\mathcal G_{v_1\otimes v_2}\leq (H_1\times H_2)^{(\varphi_1,\varphi_2)}$.

\medskip

(ii) Take $v \in (U_1 \times U_2)_{(H_1 \times H_2)^{(\varphi_1, \varphi_2)}}$ and decompose it as
$$
v  = \sum_{i=1}^n v_i \otimes e_i.
$$
For any $g := (k_1,1_{H_1},\varphi_1(k_1)) \in (H_1 \times H_2)^{(\varphi_1, \varphi_2)}$, one has
$$
g \cdot v = \sum_{i=1}^n  e^{\varphi_i(k_i)} T_{h_1} v_i \otimes e_i = \sum_{i=1}^n v_i \otimes e_i,
$$
hence, $v_i \in U_1^{K_1^{\varphi_1}}$
for any $i = 1,...,n$, i.e.,
$(\mathcal G_1)_{v_i} \leq H_1^{\varphi_1}$. A similar argument shows that $(\mathcal G_2)_{v_i} \leq H_2^{\varphi_2}$.

\end{proof}

A direct consequence of Lemma \ref{lem:factor-product} is the following:
\begin{corollary}\label{for:two observations}
Under the notations of Lemma \ref{lem:factor-product}, assume that $H_1^{\varphi_1}$ and $H_2^{\varphi_2}$ are maximal isotropies in $U_1$ and $U_2$, respectively, and  $N_{\mathcal G_i}(H_i)= H_i$.
Then, $(H_1 \times H_2)^{(\varphi_1,\varphi_2)}$
is a maximal isotropy in the tensor product.

\end{corollary}

\begin{proof}
From Lemma \ref{lem:factor-product}(i) it immediately follows that $(H_1 \times H_2)^{(\varphi_1,\varphi_2)}$  is an isotropy. Assume for contradiction that it is submaximal. Then

\begin{equation}\label{inclusion2}
(H_1 \times H_2)^{(\varphi_1,\varphi_2)} \lneq (\widetilde H_1 \leftidx{}{_{K_1}}{\times}_{K_2} \widetilde H_2)^{(\widetilde\varphi_1,\widetilde\varphi_2)},
\end{equation} where
\begin{equation}\label{inclusion!}
\widetilde H_i \vartriangleright K_i > H_i
\end{equation} and  $\widetilde{\varphi_i}$ is an extension of $\varphi_i$ to $\widetilde H_i$. 
If $K_i\gneq H_i$ then by Lemma \ref{lem:factor-product}(ii), $K_i^{\varphi_i}$ is contained in an isotropy, which contradicts maximality of $H_i^{\varphi_i}$.
 Since, by assumption, $N_{\mathcal G_i}(H_i) = H_i$ it follows from \eqref{inclusion!} that $\widetilde H_i = H_i$ which contradicts \eqref{inclusion2}.
\end{proof}

\medskip
Returning to the particular situation, where $\mathcal G_i = \mathbb D_N$ and $U_i = (\mathcal U^d_{t_i})^c$, we have the following:
\begin{lemma}\label{YTM2}
If $H_1^{\varphi_1}, H_2^{\varphi_2}$ and $H_3^{\varphi_3}$ are  maximal isotropies in $(\mathcal U_{{t}_1}^d)^c$, $(\mathcal U_{{t}_2}^d)^c$ and $(\mathcal U_{{t}_3}^d)^c$ respectively, then $(H_1\times H_2 \times H_3)^{\varphi_1\varphi_2\varphi_3}$ is a maximal isotropy in $(\mathcal{V}_{\bf t}^d)^c$.
\end{lemma}

\begin{proof}
Let us begin by observing that the maximal orbit types in $(\mathcal U_{{t}_i}^d)^c$ are either $\mathbb D_N\times \{1\}$ in the case when $t_i=0$, or $(\mathbb Z_N^{t_i}),(\mathbb D_1^+)$ and $(\mathbb D_1^-)$ if $t_i \neq 0$. By assumption, $N$ is {\it odd}, therefore $N(\mathbb D_1)=\mathbb D_1$. This means that Lemma \ref{lem:factor-product} and Corollary \ref{for:two observations} exclude the possibility that $(H_1\times H_2 \times H_3)^{\varphi_1\varphi_2\varphi_3}$ is not a maximal isotropy except in the case  when $(H_1\times H_2 \times H_3)^{\varphi_1\varphi_2\varphi_3}= (\mathbb Z_N\times \mathbb Z_N \times \mathbb Z_N)^{t_1t_2t_3}$. However, it follows from Lemma \ref{lem:factor-product} that the only candidate for a twisted subgroup of $\mathbb D_N\times \mathbb D_N\times \mathbb D_N\times S^1$, which is an isotropy in $(\mathcal{V}_{\bf t}^d)^c$ and contains $(\mathbb Z_N\times \mathbb Z_N \times \mathbb Z_N)^{t_1t_2t_3}$, is of the form $\mathcal H^\varphi$, where
$$
\mathcal H = \{(g_1,g_2,g_3)\in \mathbb D_N\times \mathbb D_N \times \mathbb D_N : \psi(g_1)=\psi(g_2)=\psi(g_3)\}
$$
 (here $\psi: \mathbb D_N \to \mathbb Z_2$ is the homomorphism with kernel $\mathbb Z_N$). On the other hand, it can be easily   seen that any vector $x\in (\mathcal{V}_{\bf t}^d)^c$ which is fixed by $(\mathbb Z_N\times \mathbb Z_N \times \mathbb Z_N)^{t_1t_2t_3}$ cannot be fixed by the element $(\kappa,\kappa,\kappa,e^{i\theta})$ for any $e^{i\theta}\in S^1$, hence $(\mathbb Z_N\times \mathbb Z_N \times \mathbb Z_N)^{t_1t_2t_3}$ is also maximal.
\end{proof}

\medskip

Being motivated by Lemma \ref{YTM2}, we introduce the following notations:
$$S(t_i)= \begin{cases}
\{(\mathbb D_N\times\{1\})\}  &\text{if} \; t_i=0\\
\{ (\mathbb Z_N^{t_i}) ,\; (\mathbb D_1^+), \; (\mathbb D_1^-)\} & \text{if} \; t_i \neq 0
\end{cases}
$$
and

\begin{equation}\label{desc_iso}
\mathcal S({\bf t})= \{(H_1\times H_2\times H_3)^{\varphi_1\varphi_2\varphi_3}: (H_i^{\varphi_i})\in S(t_i)\}
\end{equation}

\subsection{Stability analysis of system  \eqref{vdPl}}\label{subsec:Lyapunov-coef}
In this subsection, we analyze stability of Hopf branches of periodic solutions to system
 \eqref{vdPl}. This information will be used  later
for the stability analysis of the fully synchronized Hopf branches of periodic solutions to system \eqref{nvdPl_2d}.
\begin{lemma}\label{teoorigin}
For the parameter $a$ crossing zero, system  \eqref{vdPl} undergoes a
{\bf supercritical} Hopf bifurcation.
\end{lemma}
\begin{proof}
By inspection, the eigenvalues of the linearization of \eqref{vdPl} at the origin have the form $\lambda = {a \pm \sqrt{a^2 - 4b} \over 2}$,
in particular, the Hopf bifurcation takes place when $a$ crosses zero.
To analyze the character of the bifurcation, take the linear
change of coordinates:
\begin{align}\label{changecoord}
\left\{
\begin{array}{l}
\tilde{x}=\displaystyle{\frac{a}{\sqrt{4b-a^2}}x - \frac{2}{\sqrt{4b-a^2}}}y\\
\\
\tilde{y}=x
\end{array}
\right. \\
\hspace{0.01cm}\Rightarrow
\left\{
\begin{array}{l}
\dot{\tilde{x}}=\displaystyle{\frac{a}{2}\tilde{x}-\frac{\sqrt{4b-a^2}}{2}} \tilde{y}  \; \displaystyle{-\frac{a\tilde{y}^2-a\tilde{y}^3}{\sqrt{4b-a^2}}}
\\
\\
\dot{\tilde{y}}=\displaystyle{\frac{\sqrt{4b-a^2}}{2}}\tilde{x}+\frac{a}{2}\tilde{y} \; \displaystyle{-\tilde{y}^2-\tilde{y}^3}.
\end{array}
\right.
\end{align}
By direct computation, near $a = 0$, one has:
\begin{equation}\label{Lyap-direction}
\frac{\partial} {\partial a } \Big(\text{Re}( \lambda)\Big)  =\frac{1}{2} > 0 \;\; \text {and}\;\; l_1 =-{3 \over 8} < 0,
\end{equation}
where $l_1$ stands for the first Lyapunov coefficient. Combining \eqref{Lyap-direction} with \citep{Guckenheimer}, Theorem 3.4.2, completes the proof.


\end{proof}

\section{Occurrence of Hopf bifurcations}\label{sect:occurrence-result}

\subsection{Abstract result}
Let $G$ be a finite group and $U$ be an orthogonal $G$-representation which admits a $G$-isotypical decomposition
\begin{equation}\label{gen_iso_decom}
U= U_0\oplus \cdots U_k,
\end{equation}
where $U_j$ is modeled on the irreducible representation $\mathcal U_j$. We will denote by $\mathcal U_j^c$ the complexification of $U_j$ which is a $G\times S^1$-representation.


\noindent Suppose $f:\mathbb R \oplus U \to U$ is a $C^1$-smooth function and consider the system

\begin{equation}\label{abst_sys}
\dot x(t)= f(\alpha,x).
\end{equation}

\begin{definition}\label{iso_cen}
We will say that $(\alpha_0,0)$ is an isolated center with limit frequency $\beta_0$ of \eqref{abst_sys} if:

\smallskip

\noi(i) $(\alpha_0,0)$ is a center of \eqref{abst_sys} with limit frequency $\beta_0$, that is $D_xf(\alpha_0,0)$ admits $i\beta_0$ as a purely imaginary eigenvalue;

\noi(ii) $(\alpha_0,0)$ is the only center in a neighborhood of  $(\alpha_0,0)$ in $\mathbb R \oplus U $.
\end{definition}


We are now in a position to formulate the abstract occurrence result which we will apply to the system considered in this paper.

\begin{theorem}[cf. \citep{AED}]\label{abs_therom}
Suppose $f$ in system \eqref{abst_sys} satisfies the following conditions:

\begin{itemize}
\item[(P1)] $f$ is a $C^1$-smooth equivariant map (we assume $G$ acts trivially on $\mathbb R$);

\item[(P2)] $f(\alpha,0)=0$ for all $\alpha \in \mathbb R$;

\item[(P3)] $(\alpha_0,0)$ is an isolated center for \eqref{abst_sys} (cf. Definition \ref{iso_cen});

\item[(P4)] $\det D_xf(\alpha_0,0) \neq 0$;

\item[(P5)] $D_xf(\alpha_0,0)_{|U_j}$ decreases stability as $\alpha$ passes $\alpha_0$, while the stability of $D_xf(\alpha_0,0)_{|U_k}$ does not increase for any $U_k$ (cf \eqref{gen_iso_decom}).
\end{itemize}
Then, for every  maximal orbit type $(H^\varphi)$ in $\mathcal U^c_j$, there exist $|(G\times S^1)/H^\varphi|_{S^1}$ branches of non-constant periodic solutions of \eqref{abst_sys} bifurcating from the origin  with (spatio-temporal) symmetry $(H^\varphi)$ and limit period $2\pi/\beta_0$ (cf. Definition \ref{spa_temp}). Here $|(G\times S^1)/H^\varphi|_{S^1}$ is the number of $S^1$-orbits in the space $(G\times S^1)/H$.
\end{theorem}

\begin{remark}
{\rm Using the concept of isotypical crossing number  one can relax condition (P5) (see, for example, \citep{AED,BK-chapter,Sliding}).}
\end{remark}

\subsection{Proofs of Theorems \ref{thm_vdp} and \ref{thm_vdpl}}\label{Proof-Theor-occurr}

To detect the occurrence of the equivariant Hopf bifurcation in systems \eqref{nvdP_2d} and \eqref{nvdPl_2d} and to classify symmetric properties of the resulting branches, we will combine the equivariant spectral data collected in Subsections \ref{sec_decop}--\ref{iso_sec} with Theorem \ref{abs_therom}.
\medskip

\noindent{\it (a) Proof of Theorem \ref{thm_vdp}}:
We begin by observing that conditions (P1) and (P2) are obvious. It follows immediately from \eqref{Dn_blocks}--\eqref{eig_d_syst} and $b>0$  that system \eqref{nvdP_2d} can only have a center $(a,0)$ when $a=0$.
Also, since $b,K_{\bf t}>0$, formula \eqref{Dn_blocks} implies (P4). Finally,
all eigenvalues cross the imaginary axis in the same direction (see \eqref{eig_d_syst}), meaning that (P5) is also satisfied. Combining this with the description of isotropies in Section \ref{iso_sec}(i) completes the proof of Theorem \ref{thm_vdp}.

\medskip
\noindent{\it (b) Proof of Theorem \ref{thm_vdpl}}: Observe that (P1) and (P2) are obvious. Plugging $\lambda = i\omega$ into the characteristic equation of matrix $L_{\bf t}(a)$ (cf. \eqref{matL_t}) and seperating real and imaginary parts yields that all centers $(a^z_{\bf t},0)$ of system \eqref{nvdPl_2d} are given by  \eqref{eq:bifur-points}.
From this (P3) follows immediately, while (P4) is provided by $b>0$. Finally, differentiating \eqref{eig_z_sys} with respect to $a$ at $a = a^z_{\bf t}$ shows that all eigenvalues cross the imaginary axis in the same direction.  Combining this with the description of isotropies in Section \ref{iso_sec}(ii) (in particular, formula \eqref{desc_iso}) completes the proof of Theorem \ref{thm_vdpl}.

\section{Stability of bifurcating branches}\label{sect-stability}

\subsection{Proof of Theorem \ref{Dn_instability}}
Recall that if the equilibrium is unstable then any bifurcating branch of non-constant periodic solutions will also be unstable. Since the eigenvalues of the linearization of system
 \eqref{nvdP_2d} at the origin are given by $\sqrt{-bK_{\bf t}}$, it is easy to see that if $K_{\bf t}$ is negative for some {\bf t}, then the theorem is proved. Take $k_1,k_2,k_3$ provided by the hypothesis of  Theorem \ref{Dn_instability}. Clearly, $K_{\bf t} < 0 $ for ${\bf t} = \big(k_1\big(\frac{n-1}{2}\big), k_2\big(\frac{n-1}{2}\big), k_3\big(\frac{n-1}{2}\big)\big)$ which completes the proof.

\medskip

\subsection{Proof of Theorem \ref{Zn_stab}} It is clear from formulas \eqref{H_t} and  \eqref{eig_z_sys} that the sign of $\text{Re} (\lambda_{\bf t}^z(a))$ is given by the sign of $H_{\bf t}(a)$. Observe that the
largest value of $H_{\bf t}(a)$ is achieved by the value of ${\bf t}$ specified in the Table \ref{tabelaBifurca} and changes its sign at $a = a^{*}$. Therefore, for any $\bf{t}$ and $a < a^*$, one has
$\text{Re} (\lambda_{\bf t}^z(a)) < 0$. Therefore, for $a < a^*$, the equilibrium is stable while for  $a > a^*$, the equilibrium is unstable.

In the case when $\delta, \zeta, \varepsilon < 0$, only one pair of eigenvalues crosses the imaginary axis at $a^*=0$. Observe that the central manifold coincides with the central space. To complete the proof we combine Lemma \ref{teoorigin} with the fact that all other eigenvalues have negative real part (cf. \citep{Guckenheimer}, Theorem 3.4.2).


\section{Existence of periodic solutions with prescribed period and symmetry
} \label{sec:Existence}

\subsection{Abstract result.}

To prove Theorem \ref{Existence-Dihedral}, we will use a slight modification of the main result from \citep{AED}, Chapter 12 (cf. Theorem 12.7). To begin with, we need the following
\begin{definition}\label{def:lambda-max-types}
{\rm Let $G$ be a finite group and let $V$ be a real orthogonal $G$-representation. Assume $A : V \to V$ is an equivariant linear operator and $\lambda \in \sigma(A^c)$.
Take the eigenspace $E(\lambda) \subset V^c$
and denote by $\mathcal O(\lambda)$ the set of all maximal $G \times S^1$-orbit types occurring in  $E(\lambda)$.
}
\end{definition}
Let $V :=\mathbb R^n$ be an orthogonal permutational $G$-representation and
consider the system
\begin{equation}\label{syst:VDP-Hirano_Ryb}
\begin{cases}
\dot{x}_i = \nu\Big(ax_i - {x^3_i \over 3}\Big) - \sum_{j=1}^n C_{ij}y_j\\
\dot{y}_i = b x_i
\end{cases} \quad (i = 1,\ldots, n).
\end{equation}

\begin{theorem}\label{theom-abstract-Hirano-Ryb} Assume $C$ is a non-singular $G$-equivariant symmetric matrix.
Then, for each real positive $p \not\in \{\frac{2\pi (2k-1)}{\sqrt{\mu}} \; : \; \mu \in \sigma^+(b C), \;  k \in \mathbb N \}$ and each $\mu \in \sigma(bC)$ satisfying $0 < (\frac{2\pi}{p})^2 <  \mu$, there exists $\nu > 0$ such that for every $(H^\varphi) \in \mathcal O(\mu)$, system  \eqref{syst:VDP-Hirano_Ryb}  admits $\frac{|G|}{|H|}$ $p$-periodic solutions with minimal symmetry $(H^varphi)$.
\end{theorem}

\subsection {Proof of Theorem \ref{Existence-Dihedral}}
The linearization of system  \eqref{syst:VDP-Hirano_Ryb} at the origin has the form:
$$
\mathfrak A=
\begin{bmatrix}
\nu a \text{Id} & -C\\
b \text{Id} & 0
\end{bmatrix}
$$
where Id stands for the $n\times n$-identity matrix. Since in the case of system \eqref{nvdP_2d}, the isotypical decomposition of $V$ contains only irreducible representations of real type, then $\mathfrak A_{|V_i}$ is of the form
$$
\begin{bmatrix}
\nu a  & -\lambda\\
b  & 0
\end{bmatrix} \otimes \text{Id}_{|V_i}
$$
where $\lambda$ is an eigenvalue of $C_{|V_i}$. It follows from Remark \ref{prim_decom} and formula \eqref{Dn_blocks} that
$$
\sigma(C)=\{K_{\bf_t}: 0<t_i<(N-1)/2\}.$$
Also note that $\mathcal O(K_{\bf t})= \mathcal S({\bf t})$. Combining this with Theorem \ref{theom-abstract-Hirano-Ryb} completes the proof of Theorem \ref{Existence-Dihedral}.

\section{Appendix}

If $W$ is a $G$-representation, then for any function $x:S^1 \to W$ the spatio-temporal symmetry of $x$ is a group $\mathfrak H < G \times S^1$ such that $g\cdot x(t-\theta)=x(t)$ for any $t\in \mathbb R/2\pi\mathbb Z \simeq S^1$ and any $(g,e^{i\theta}) \in \mathfrak H$. If $x$ is  non-constant, then its spatio-temporal symmetry group will have the structure of a graph of a homomorphism from some subgroup $H < G$ to $S^1$. In our general discussion, we used the following notation:
$$
H^\varphi := \{(h,\varphi(h)\;:\;h\in H)\}.
$$
We call this a {\it twisted} symmetry group with twisting homomorpism $\varphi$. If the domain of the twisting homomorphism is a direct product of groups, we can describe the twisting homomorphism by its restrictions to each of the components. Therefore, we use  the following notation:
\begin{equation}
 \label{prod_twist}
 \begin{aligned}
(H_1\times H_2\times H_3)^{\varphi_1\varphi_2\varphi_3} := &\{(h_1,h_2,h_3,\varphi_1(h_1)\varphi_2(h_2)\varphi_3(h_3))\\ &\;:\; (h_1,h_2,h_3)\in H_1\times H_2\times H_3\}
\end{aligned}
\end{equation}

  Given two groups $\mathscr G_1$ and $\mathscr G_2$, to describe subgroups of  $\mathscr G_1\times \mathscr G_2$,
  define the projection homomorphisms:
  \begin{align*}
    \pi _{1}&:\mathscr G_1\times\mathscr G_2\to \mathscr G_1,\quad\pi_1(g_1,g_2)=g_1;\\
    \pi _{2}&:\mathscr G_1\times\mathscr G_2\to \mathscr G_2,\quad\pi_2(g_1,g_2)=g_2.
  \end{align*}
  The following result, being a reformulation of the well-known
  Goursat's Lemma (cf.~\citep{Goursat}),
  provides the desired description of subgroups $\mathscr H$ of the
  product group $\mathscr G_1\times \mathscr G_2$.
  \begin{theorem}\label{th:product-1}
    Let  $\mathscr H$ be a subgroup of the product group
    $\mathscr G_1\times \mathscr G_2$. Put $H:=\pi_1(\mathscr H)$
    and $K:=\pi_2(\mathscr H)$. Then, there exist a group $L$ and
    two epimorphisms $\vp :H\rightarrow L$ and $\psi :K\rightarrow L$,
    such that
    \begin{equation}\label{eq:product-rep}
      \mathscr H=\{(h,k)\in H\times K: \vp(h)=\psi(k)\}.
    \end{equation}
    \end{theorem}

For the needs of our paper, it is enough to characterize such subgroups up to the kernels of the homomorphisms $\varphi$ and $\psi$.  For this reason, we put
\begin{equation}\label{amal_not}
\mathscr H =: H_1 \leftidx{}{_{K_1}}{\times}_{K_2}  H_2
\end{equation}
where $K_1$ is the kernel of $\varphi$ and $K_2$ is the kernel of $\psi$.

\medskip 

In the particular case of $\mathbb D_N$, we denote by $\mathbb Z_N$ the subgroup generated by $e^{\frac{2\pi i}{N}}$, and by $\mathbb D_1$ the subgroup generated by $\kappa$. Furthermore, we put
\begin{eqnarray*}
\mathbb Z_N^{t_i} := \mathbb Z_N^\varphi &\text{where} & \varphi:e^{k\frac{2\pi i}{N}}\mapsto e^{kt_i\frac{2\pi i}{N}}\\
\mathbb D_1^+ := \mathbb D_1^\varphi &\text{where} & \varphi:  \kappa \to 1\\
\mathbb D_1^- := \mathbb D_1^\varphi &\text{where} & \varphi:  \kappa \mapsto -1
\end{eqnarray*}
We combine this with the notation for twisted groups given in \eqref{prod_twist} in the obvious way, for example:
$$
(\mathbb Z_N\times \mathbb D_1\times \mathbb Z_N)^{(t_1,-,t_2)}:=(\mathbb Z_N\times \mathbb D_1\times \mathbb Z_N)^{\varphi_1\varphi_2\varphi_3}
$$
where
$\varphi_1:e^{k\frac{2\pi i}{N}}\mapsto e^{kt_1\frac{2\pi i}{N}}$, $\varphi_2:\kappa \mapsto -1$ and $\varphi_3:e^{k\frac{2\pi i}{N}}\mapsto e^{kt_2\frac{2\pi i}{N}}$.


\bibliographystyle{siam}
\bibliography{torus2_bib}{}
\end{document}